\documentclass[11pt]{amsart}
\usepackage{amsmath,amssymb,amsfonts,amsthm,amscd,indentfirst}
\usepackage{amsmath,amssymb,amsfonts,amsthm,amscd, color}
\usepackage{indentfirst}
\usepackage{hyperref}

\numberwithin{equation}{section}

\newtheorem{prop}{Proposition}[section]
\newtheorem{theo}[prop]{Theorem}

\newtheorem{coro}[prop]{Corollary}

\newtheorem{defi}[prop]{Definition}

\theoremstyle{remark}
\newtheorem{rema}{Remark}[section]

\def\and{\quad{\rm and}\quad}

\def\p{\partial}
\def\m{\mathfrak{m}}
\def\E{\mathcal{E}}

\def\<{\langle}
\def\>{\rangle}

\def\Sh{\Sigma_{_H}}
\def\Sht{\Sh^{(t)}}
\def\be{\begin{equation}}
\def\ee{\end{equation}}

\usepackage{graphicx}

\begin{document}

\title{Rigidity of Riemannian Penrose inequality with corners and its implications}

\author[Siyuan Lu]{Siyuan Lu}
\address{Department of Mathematics and Statistics, McMaster University, 
1280 Main Street West, Hamilton, ON, L8S 4K1, Canada.}
\email{siyuan.lu@mcmaster.ca}

\author[Pengzi Miao]{Pengzi Miao}
\address{Department of Mathematics, University of Miami, Coral Gables, FL 33146, USA.}
\email{pengzim@math.miami.edu}

\begin{abstract}
Motivated by the rigidity case in the localized Riemannian Penrose inequality, 
we show that suitable singular metrics attaining the optimal value 
in the Riemannian Penrose inequality is necessarily smooth in properly specified coordinates. 
If applied to hypersurfaces enclosing the horizon in a spatial Schwarzschild manifold, the result
gives the rigidity of isometric hypersurfaces with the same mean curvature. 
\end{abstract}

\maketitle 

\markboth{Siyuan Lu and Pengzi Miao}{Rigidity of Riemannian Penrose inequality with corners}

\section{Introduction}

In this paper, we prove a rigidity theorem for a pair of Riemannian manifolds with nonnegative scalar curvature,
with boundary. The theorem may be viewed as the rigidity part of the Riemannian Penrose inequality on manifolds 
with corners along a hypersurface. 
We assume all manifolds have dimension $ n \le 7$. 

\begin{theo} \label{thm-main}
Let $ (\Omega^n, g_{_\Omega})$, $ (N^n, g_{_N})$  be a compact manifold, an asymptotically flat manifold,
with nonnegative scalar curvature, 
with boundary $\p \Omega$, $ \p N$,  respectively.
Suppose the boundaries $ \p \Omega$ and $ \p N$ satisfy the following
\begin{enumerate}
\item[(i)]  $\p \Omega$  is the disjoint union of two pieces, $\Sh$ and $ \Sigma $, 
where $\Sh$ has zero mean curvature and is strictly outer-minimizing in $(\Omega, g_{_\Omega})$;
\item[(ii)]  $\p N$ is outer-minimizing in $(N, g_{_N})$; and
\item[(iii)] $\Sigma$ is isometric to $ \p N$, with the induced metrics; and under the isometry, 
$ H_{_{\Omega} } \ge H_{_N}$, where $ H_{_\Omega}$ is the mean curvature of $ \Sigma$ in 
$(\Omega, g_{_\Omega})$ with respect to the outward normal, and $ H_{_N}$ is the mean curvature of $ \p N$ 
in $(N, g_{_N})$ with respect to the infinity pointing normal.
\end{enumerate}
Let $ \m (g)$ be the mass of $(N, g)$ and let $ | \Sh |$ be the area of $ \Sh$ in $(\Omega, g_{_\Omega})$.
Suppose
$$ \m (g) = \frac12 \left( \frac{ |\Sh|}{  \omega_{n-1} } \right)^\frac{n-2}{n-1} , $$
where $ \omega_{n-1}$ is the area of the standard round $(n-1)$-dimensional sphere. Then
\begin{itemize}
\item $\Sigma $ and $ \p N$ have the same second fundamental forms; 
\item $\Sigma $ (and hence $\p N$) isometrically embeds in a spatial Schwarzschild manifold
$$ (\mathbb{M}_m, g_m) = \left( \left\{ x \in \mathbb{R}^{n}  :  |x| \ge  \left( \frac{m}{2} \right)^\frac{1}{n-2}  \right\}, 
\left( 1 + \frac{m}{2}  | x |^{2-n}  \right)^\frac{4}{n-2} g_{0} \right)  $$
with mass $ m = \m (g)$. Here $g_0$ is the Euclidean metric on $\mathbb{R}^n$.
Moreover, the image of this embedding and the Schwarzschild horizon $ \p \mathbb{M}_m$ 
 enclose a bounded domain $\Omega_m$ in $ \mathbb{M}_m$;  and
\item $(\Omega, g_{_\Omega})$ is isometric to $ (\Omega_m, g_m)$ and $(N, g_{_N})$ is isometric to the complement of 
$\Omega_m$ in $(\mathbb{M}_m, g_m)$.
\end{itemize}
\end{theo}

The condition that $ \Sh$ is strictly outer-minimizing in $(\Omega, g_{_\Omega})$ means that any hypersurface $ \Sigma' $ 
in $ \Omega$, which encloses $\Sh$,  has area strictly greater than $|\Sh|$. Similarly, 
$ \p N $ is outer-minimizing in $(N, g_{_N})$ means that any hypersurface $ \Sigma'' $ 
in $N$, which encloses $\p N$,  has area greater than or equal to the area of $\p N$.

In Theorem \ref{thm-main}, we state the conclusion in a geometric manner. 
From a more analytic perspective, our proof of Theorem \ref{thm-main} indeed
gives a regularity result that asserts suitable singular metrics realizing the optimal value 
in the Riemannian Penrose inequality is smooth in properly specified coordinates.  
We formulate this conclusion in the following remark.

\begin{rema} \label{rem-main}
Under the assumptions of Theorem \ref{thm-main}, one can consider a new differentiable manifold 
obtained from $\Omega$ and $N$ as follows.

Let $ U_+ $ be a Gaussian tubular neighborhood of $\p N$ in $(N, g_{_N})$ so that 
$U_+ $ is diffeomorphic to $  [ 0 , \epsilon ) \times \p N $ for some $ \epsilon > 0 $
and  $ g_{_N} = d t^2 + g_t^+$ in $U_+$. 
Similarly, let $U_-$ be a Gaussian tubular neighborhood of $\Sigma $ in $(\Omega , g_{_\Omega})$ so that 
$U_- $ is diffeomorphic to $  ( - \epsilon ,0] \times \Sigma  $ 
and $ g_{_\Omega} = d t^2 + g_t^-$ in $U_-$.
Here $\{ g_t^+ \}_{t \ge 0}$, $\{ g_t^- \}_{t \le 0 }$ denote a family of metrics on $ \p N$, $ \Sigma$, respectively. 
Since $ \Sigma$ is  isometric to $ \p N$, identifying $ \Sigma $ with $ \p N$ via the given isometry, 
one may assume $ g_t^+ = g_t^- $ at $ t =  0 $. 
Let $M$ be the topological manifold obtained by gluing $\Omega $ and $N$ so that 
$ \Sigma $ and $ \p N$ are identified via the given isometry. 
Define the differentiable structure on $M$ so that it is determined by the open covering consisting of 
$ \{ \Omega, N, U \} $, where 
$U = U_- \cup U_+ = (-\epsilon , \epsilon ) \times \Sigma $.

On this differentiable manifold $M$, consider a Lipschitz metric
$ g $ given by 
$$ g = g_{_\Omega} \ \text{on} \  \Omega, \ \ \text{and} \ \ 
 g  = g_{_N} \ \text{on}  \ N  . $$
If $(\Omega, g_{_\Omega})$ and $(N, g_{_N})$ satisfy $  \m (g) = \frac12 \left( \frac{ |\Sh|}{  \omega_{n-1} } \right)^\frac{n-2}{n-1} $, 
Theorem \ref{thm-rigidity} in Section \ref{sec-pf} shows that 
$g$ is smooth across $\Sigma$  and 
the manifold $(M^n, g)$ is isometric to $ (\mathbb{M}_m, g_m)$ with $m = \m (g)$.
\end{rema}

Our main motivation to consider Theorem \ref{thm-main} is the rigidity case of the localized Riemannian
Penrose inequality studied in \cite{LM17}. As a corollary of Theorem \ref{thm-main}, equality in Theorem 1.1 of \cite{LM17} holds 
if and only if the compact manifold in that setting is isometric to the domain enclosed by the image of the relevant 
isometric embedding into $ (\mathbb{M}_m, g_m)$ and the Schwarzschild horizon $\p \mathbb{M}_m$.

\medskip

Another motivation for us is to study the rigidity of isometric hypersurfaces with the same mean curvature in $(\mathbb{M}_m, g_m)$.
For convenience, we make the following definition. 

\begin{defi}
In an ambient Riemannian manifold, we say two hypersurfaces $\Sigma$ and $ \tilde \Sigma$ are H-isometric if 
there is an isometry $F: \Sigma \to \tilde \Sigma$ such that $\tilde H ( F (p) ) = H (p)$, $\forall \, p \in \Sigma$.
Here $ H$, $ \tilde H$ are the mean curvatures of $\Sigma$, $ \tilde \Sigma$, respectively. 
\end{defi}

A corollary of Theorem \ref{thm-main} and Remark \ref{thm-main} is 

\begin{coro} \label{cor-iso-hypersurface}
Let $ \Sigma \subset (\mathbb{M}_m, g_m)  $ be a closed hypersurface enclosing the horizon $\p \mathbb{M}_m$
in a spatial Schwarzschild manifold with mass $ m >0$. Suppose $\Sigma$ is outer-minimizing. 
If $ \tilde \Sigma $ is another hypersurface enclosing $\p \mathbb{M}$ which is H-isometric to $\Sigma$, then 
$\tilde \Sigma= \Sigma$ up to a rigid motion of $(\mathbb{M}_m, g_m)$.
\end{coro}

The rigidity of H-isometric hypersurfaces seems an interesting question that 
links the isometric embedding problem in Riemannian geometry to the context of 
quasi-local mass in general relativity. In the Euclidean space $\mathbb{R}^3$, a classic result of Cohn-Vossen \cite{CV} shows 
that convex surfaces are rigid. Recently Li and Wang \cite{LW} gave counterexamples which illustrate the lack 
of rigidity even for convex surfaces if the ambient manifold is not a space form.
On the other hand, relativistic consideration in relation to the Bartnik mass \cite{Bartnik}
seems to suggest H-isometric surfaces can be rigid if the ambient manifold is static.
In a spatial Schwarzschild manifold $(\mathbb{M}_m, g_m)$, 
such a rigidity was shown by Chen and Zhang \cite{CZ} for suitable convex surfaces in dimension $3$. 
Analogous results were given by Li, Wang and the second author \cite{LMW18} among starshaped hypersurfaces in 
an $(\mathbb{M}_m, g_m)$ of general dimensions.

\medskip

We now explain the proof of Theorem \ref{thm-main}. 
A main tool we use is Bray's proof of the Riemannian Penrose inequality (RPI)
in dimension three \cite{B} and Bray-Lee's proof of the RPI for dimensions $ n \le 7$ \cite{BL}.
We use the flow produced in \cite{B, BL} to perturb the singular metric $g$ constructed in Remark \ref{rem-main}, and 
analyze the case that the mass has a zero derivative.  This  relates to another ingredient in the proof,
which is a revisit of the rigidity case of the Riemannian positive mass theorem with corners along a hypersurface. 

\begin{theo} \label{thm-main-2}
Let $ (\Omega^n, g_{_\Omega})$, $ (N^n, g_{_N})$  be a compact manifold, an asymptotically flat manifold,
with nonnegative scalar curvature, with boundary $\p \Omega$, $ \p N$,  respectively.
Suppose $\p \Omega$ is isometric to $ \p N$, and under the isometry, the mean curvatures satisfy 
$H_{_\Omega} \ge H_{_N}$. 
If $\m (g_{_N}) = 0 $, then 
\begin{itemize}
\item $\p \Omega$ and $ \p N$ have the same second fundamental forms; and
\item the manifold $(M, g)$, constructed by gluing $(\Omega, g_{_\Omega})$ and $(N,  g_{_N})$ 
in Remark \ref{rem-main}, is smooth and is isometric to $ (\mathbb{R}^n, g_0)$.
\end{itemize}
\end{theo}

We give an account of previously known results that relate to Theorems \ref{thm-main} and  \ref{thm-main-2}.
Both rigidity questions are tied to the quasi-local mass problem (see \cite{ST, LM17} for instance).
If the manifolds are spin, Shi and Tam proved Theorem \ref{thm-main-2} in \cite{ST}.
Without the spin assumption, McFeron and Sz\'{e}kelyhidi \cite{MS} proved a variation of Theorem \ref{thm-main-2},
from which Theorem \ref{thm-main-2} is derived. 
We will explain how the results in \cite{MS} implies Theorem \ref{thm-main-2} in Proposition \ref{prop-refined-MS}.
In the case of Riemannian Penrose inequality with corners, 
the rigidity part was studied by Shi, Wang and Yu \cite{SWY18} in $3$-dimension, and the manifolds were shown to be static with 
zero scalar curvature and, under an additional geometric condition, 
$(\Omega, g_{_\Omega})$ was proven to be isometric to a region in $(\mathbb{M}_m, g_m)$.
Theorem \ref{thm-main} was also proved by the authors \cite{LM18} for the case $ n =3$, 
in the setting of the localized Penrose inequality \cite{LM17}.

The Riemannian Penrose inequality was first proved by Huisken and Ilmanen \cite{HI01} for connected horizon, 
and by  Bray \cite{B} for general horizon, both in dimension $3$. 
In \cite{BL}, Bray and Lee established the inequality for dimension $ n \le 7 $. 
As the proof of Theorem \ref{thm-main} makes use of Bray and Lee's work  \cite{BL}, Theorem \ref{thm-main}
satisfies the same dimension assumption. 

\section{Rigidity of PMT with corners along a hypersurface} \label{sec-rigidity}

In this section, we revisit the rigidity case of the Riemannian positive mass theorem \cite{SY79, W81}
formulated on manifolds with corners along a hypersurface.

We say $(M^n, g)$ is an asymptotically flat manifold with corners along a hypersurface $S$ if the following conditions hold:
\begin{enumerate}
\item $M$ is a smooth differentiable manifold and $S \subset M$ is a compact, 
embedded two-sided hypersurface which can be disconnected;

\item $g$ is a $C^0$ metric on $M$, $g$ is smooth away from $S$, and $ (M \setminus K, g)$ is asymptotically flat 
for some compact set $K$ containing $S$; 

\item there exists a smooth open neighborhood $U$ of $S $ such that $U$ is diffeomorphic to 
$ S \times (- \epsilon, \epsilon) $ on which the metric $g$ takes the form of 
$g = d t^2 + g_t $. Here $ S =  S \times \{ 0 \}$ under this diffeomorphism, and $g_t$ 
denotes the induced metric on $ S \times \{ t \}$. Moreover, if $U_+$ denotes 
$ S \times [0 , \epsilon)$ and $ U_-$ denotes $ S \times ( - \epsilon, 0]$ in $U$, 
$g$ is smooth up to the boundary $ S $ in $U_+$ and $U_-$, respectively. 
\end{enumerate}

We emphasize that all future regularity assertions of $g$ we are about to make will be with respect to 
the differential structure specified on $U$ above. 
In geometric applications, this will not impose any restriction, because one can always 
glue two Riemannian manifolds along 
their isometric boundary $S$ in the way specified in Remark \ref{rem-main}.

Given such an $(M, g)$ with corners along $S$, we say it satisfies the mean curvature condition across $S$ 
if $ H_-  \ge H_+ $ (see \cite{Miao02}). 
Here $H_+$, $H_-$ denote the mean curvature of $ \Sigma$ in $(U_+, g)$, $ (U_-, g)$ with respect to the 
normal vectors $\p_t$, respectively. 
We note that this condition is intrinsic and it remains unchanged if one switches $t$ and $-t$.

A main observation in this section is the following rigidity statement, which is built on the result of 
McFeron and Sz\'{e}kelyhidi \cite{MS}.

\begin{prop} \label{prop-pmt-rigidity} \label{prop-refined-MS}
Let $(M^n, g)$ be an asymptotically flat manifold with corners along a hypersurface $S$.
Suppose $g$ has nonnegative scalar curvature away from $S$ and satisfies the mean curvature condition across $S$. 
If $ \m (g) = 0 $, then $S$ has the same second fundamental forms in its both sides in $M$, 
the metric $g$ is smooth, and $(M^n, g)$ is isometric to the Euclidean space $ (\mathbb{R}^n, g_0)$.
\end{prop}

\begin{proof}
Under the given assumptions, it was proved in Theorem 18 of \cite{MS} that there is a $C^{1,\alpha }$ 
diffeomorphism $\phi_0: M^n \to \mathbb{R}^n$ such that $\phi_0$ is an isometry. 
We claim this $\phi_0$ is indeed $C^{1,1}$. 
The reason is as follows. The proof in \cite{MS} considered 
the solution $\{ g(t) \}_{t > 0} $ to the usually called $h$-flow (see \cite{Simon}) with an initial condition $(M, g)$.
The properties of $ g$ ensures $g(t)$, $ t > 0$, has nonnegative scalar curvature.
If $\m ( g) = 0 $, each 
$( M \, g(t) )$ is isometric to $(\mathbb{R}^n, g_0) $ and the $h$-flow 
is  acting by diffeomorphisms. By writing $ g(t) = \phi_t^* (g_0)$, the proof of Theorem 18 on page 439 in \cite{MS} 
showed that the family of diffeomorphisms $\{ \phi_t \}$ are bounded in $C^{1,1}$, 
and has a subsequence that converges in $C^{1,\alpha}$ to a diffeomorphism $\phi_0$. 
Thought not stated in \cite{MS}, the $C^{1,1}$ bound on $ \{ \phi_t \}$ ensures that the limit $\phi_0$ is $C^{1,1}$ itself. 

Now let $  h_+$, $ h_-$ denote the second fundamental forms of $ S $ in $(U_+, g)$, $(U_-, g)$, respectively.
Let $ S_0 = \phi_0 ( S) \subset \mathbb{R}^n$, then $ S_0$ is a $C^{1,1}$ hypersurface, and hence 
has a.e. defined fundamental form.
Denote this second fundamental form by $ h_0$. 
Let $ \{ \partial_\alpha \}$ be a local frame on $ S $. Let $ \partial^{(0)}_\alpha  = \phi_{0 *} ( \partial_\alpha)$, then
$\{ \partial^{(0)}_\alpha \}$ is a local $C^{0,1}$ frame on $ S_0$. The Christoffel symbols of the induced metric on 
$S_0$ with respect to $\{ \partial^{(0)}_\alpha \}$, wherever they are defined, agree with those of the induced metric 
on $\Sigma$ with respect to $ \{ \partial_\alpha \}$ under the map $\phi_0$, because $ \phi_0 : S \rightarrow S_0$ is 
an isometry.  Also, $ \phi_0$ sends the normal vector to $ S $ to the normal vector to $ S_0$.
Thus, by the definition of second fundamental forms,  $  h_+ = \phi_0^* (h_0)$, a.e. on $ S $.
Similarly, $  h_- = \phi_0^* (h_0)$, a.e. on $ S$. These imply $ h_+ =  h_-$. 

By the proof of Theorem 2 in \cite{MS}, $ g$ is flat away from $S$. This and the fact $S$ has the same second fundamental 
forms in $U_+$ and $U_-$  imply that $g$ is smooth across $S$ in $U$. (See Lemma 4.1 in \cite{ST} for instance.)
Hence, $g$ is smooth on $M$ and $(M, g)$ is isometric to $(\mathbb{R}^n, g_0)$.
\end{proof}

Theorem \ref{thm-main-2} follows from Proposition \ref{prop-pmt-rigidity} and the construction of $(M, g)$ in Remark \ref{rem-main}.
Proposition \ref{prop-refined-MS} also implies the rigidity of H-isometric hypersurfaces in Euclidean spaces. 

\begin{coro} \label{cor-Eu}
Let $ \Sigma$ and $ \tilde \Sigma$ be two closed hypersurfaces in $\mathbb{R}^n$.
If $ \Sigma$ and $\tilde \Sigma$ are H-isometric, then they differ by a rigid motion of $\mathbb{R}^n$.
\end{coro}

It is worth of noting that there are no topological assumptions on $\Sigma$, $\tilde \Sigma$ above.
In the classic study of isometric surfaces in $\mathbb{R}^3$, results are often restricted to $2$-spheres
due to various convexity assumptions on the surface.

Next, we examine the rigidity case of Bray's mass-capacity inequality, Theorem 9 in \cite{B}, for manifolds with corners along a hypersurface.

Given an asymptotically flat manifold $(M^n, g)$ with corners along a hypersurface $\Sigma$, 
with nonempty boundary $\p M $, as the metric is Lipschitz, there exists a function $\varphi $ satisfying 
\begin{equation}
\left\{ 
\begin{split}
\Delta_g \varphi(x) =  & \ 0,  \ \mathrm{in} \ M\\
\varphi(x) = & \ 0 ,   \ \mathrm{on}  \ \p M \\
\varphi(x) \to  & \ 1, \ \mathrm{as} \ x \to \infty . 
\end{split}
\right.
\end{equation}
Standard elliptic theory shows 
$\varphi \in W^{2,p}_{loc} (M)$ for any $ p > n $, hence $ \varphi \in C^{1,\alpha}_{loc} (M)$ for any $\alpha \in (0,1)$, 
$ \varphi $ is smooth away from $\Sigma$, and $ \varphi $ is smooth up to $\p M$.
The asymptotically flatness of $g$ implies $\varphi$ satisfies 
\be
\varphi (x) = 1 - \frac{ \mathcal{E} (g) }{2  |x|^{2-n} } + o (|x|^{2-n}), \ \text{as} \ x \to \infty. 
\ee
Here $\E (g) > 0$ is a constant known as the capacity of $\p M$ in $(M, g)$. 

\begin{prop} \label{prop-mass-capacity}
Let $(M^n, g)$ be an asymptotically flat manifold with corners along a hypersurface $\Sigma$.
Suppose $g$ has nonnegative scalar curvature away from $\Sigma$ and satisfies the mean curvature condition across $\Sigma$. 
If $ M $ has nonempty boundary $\Sh$ that has zero mean curvature,
then  $$ \m (g) \ge \frac12  \E(g) . $$ 
Moreover, if  $\m (g) = \frac12  \E(g)$, then $ \Sigma$ has the same second fundamental forms in its two sides in $(M, g)$, 
$g$ is smooth across $\Sigma$, and $(M, g)$ is isometric to a spatial Schwarzschild manifold outside the horizon. 
\end{prop}

\begin{proof}
We start with a property of harmonic functions at and near the singular hypersurface $ \Sigma$.
Since $g$ is smooth up to $\Sigma$ from its both sides in $M$, the restriction of 
$\varphi$ to $ \Sigma$ is indeed smooth, and $\varphi $ is smooth up to $\Sigma$ from its both sides in $M$. 
See Proposition 3.1 and Remark 3.1 in \cite{HMT} for this claim.

We proceed by conformally deforming $g$ as in \cite{B}.  More precisely, 
we first reflect $(M, g)$ across $\Sh$ and denote the resulting manifold by $(\widetilde{M},  g)$.
Clearly, $(\widetilde M, g)$ is an asymptotically flat manifold with two ends, with corners along a hypersurface
$S = \Sigma \cup \Sh \cup \Sigma' $,
where $ \Sigma'$ is the image of $ \Sigma$ under the reflection map. 
Let $\tilde{\varphi}$ be the odd extension of $\varphi$ to $\widetilde{M}$, and  let 
\be \label{eq-Bray-cmpt}
\tilde{g}= \left(\frac{1+\tilde{\varphi}}{2} \right)^{\frac{4}{n-2}}g  \ \ \mathrm{on} \ \widetilde M . 
\ee 
At the infinity of $M$, $\m (\tilde g)$ is related to $\m (g)$ by $ \m (\tilde g) = \m - \frac12 \E (g)$.
Moreover, $(\widetilde M , \tilde g)  $ satisfies the following: 

\begin{itemize}
\item[(i)] $ \tilde g$ is smooth up to $S $ in its both sides in $\widetilde M$. Here we used the above-mentioned 
regularity of $\varphi$ at and near $\Sigma$.

\item[(ii)] If $ M_-$ denotes the image of $M$ under the reflection in $\widetilde M$
and $\Omega_- = M_- \cup \{ o \} $ denotes the one-point compactification of $M_-$ 
by including a point $o$ representing the infinity of $M_-$, then $ \tilde g$ is $ W^{1, q}_{loc} $ 
near $o$ for some $ q > n$.
This is a result of the harmonic conformal factor $ \displaystyle \frac{ 1 + \varphi}{2} \to 0 $ as $ x \to  o$. (See 
Lemma 6.1 in \cite{MMT} and Lemma 4.3 in \cite{HM} for instance.)

\item[(iii)] On $\widetilde M \cup \{o\}$, $\tilde g$ has nonnegative scalar curvature away from $ S \cup \{ o \}$,
and satisfies the mean curvature condition across $S$. 
\end{itemize}

The particular type of point singularity at $o$ does not affect the positive mass theorem.
Applying the proof of Theorem 18 in \cite{MS}, combined with the proof of 
Theorem 7.2  in \cite{ShiTam16}, one has 
$ \m (\tilde g ) \ge 0 $, and 
if $ \m (\tilde g) = 0 $, $ \widetilde M \cup \{ o \}$ is diffeomorphic to $ \mathbb{R}^n$ and 
$\tilde g$ is flat away from $ S \cup \{ o \}$.
(If $n=3$, it was shown in \cite{LiMantoulidis18} that a much weaker type of point singularity suffices.)

Suppose $ \m (\tilde g) = 0 $.
Since $\tilde g$ is flat around $o$, we can assume $\tilde g$ is smooth across $o$ by 
revising the differential structure near $o$ if needed. Precisely, this follows from Theorem 3.1 in \cite{SmithYang92}.
Now $(\widetilde M \cup \{ o \} ,  \tilde g)$ is only potentially singular at $S$ and $\m (\tilde g ) = 0 $. 
By Proposition \ref{prop-pmt-rigidity}, 
$S$ has the same fundamental forms in its two sides in $(\widetilde M, \tilde g)$, 
$\tilde g$ is smooth (with respect to the differential structure specified by the Gaussian 
neighborhoods of $S$ relative to  $\tilde g$),
and $(\widetilde M \cup \{ o \}, \tilde g)$ is isometric to $ (\mathbb{R}^n, g_0)$.

Back on $(M, g)$, as $ \varphi $ is $C^1$ across $ \Sigma$, it follows from 
the conformal change formula of second fundamental forms
that $\Sigma$ has the same fundamental forms in its two sides in $(M, g)$.
Hence, $ g$ is $C^{1,1}$ across $\Sigma$, and $\varphi $ is $C^{2,\alpha}$ in $U$.
Since $ \tilde g  $ has zero curvature in $U_+$ and $U_-$, and $ \varphi $ is $C^{2}$ in $U$, 
a calculation similar to Lemma 4.1 in \cite{ST} shows $\p_t^2 g_{t} |_{t=0}  $ in $U_+$ and $U_-$ agree at $ \Sigma$.
As a result, $g$ is $C^{2,1}$ across $\Sigma$ and $\varphi $ is $C^{3,\alpha}$ in $U$.
Repeating this argument, we have $g$ and $\varphi$ are smooth in $U$.
The same argument also shows $ \Sh$ is totally geodesic in $(M, g)$, 
$g$ and $\tilde \varphi$ are smooth near $\Sh $ in $\widetilde M$.

To complete the proof, by \eqref{eq-Bray-cmpt}, $ \tilde \psi = \frac{2}{ 1 + \tilde \varphi} > 0 $ is harmonic on 
$(\widetilde M, \tilde g)$ that is isometric to $(\mathbb{R}^n \setminus \{0\}, g_0)$. 
Therefore, $\frac{2}{ 1 + \tilde \varphi } = 1 + \frac{m}{ 2 |x|^{n-2} } $
for some constant $m >0$.
We conclude that $(M, g)$ is isometric to a spatial Schwarzschild manifold outside its horizon. 
\end{proof}

\section{Proof of Theorem \ref{thm-main} and Corollary \ref{cor-iso-hypersurface}} \label{sec-pf}

We use Proposition \ref{prop-mass-capacity} and the conformal flow in \cite{B, BL} to prove the following theorem, which 
implies Theorem \ref{thm-main}.
We emphasize that the conformal flow in \cite{B, BL} is used only to produce a perturbation of the metric, 
thus we do not require the long term behavior of the flow if the initial data is a manifold with corners.

\begin{theo} \label{thm-rigidity}
Let $(M^n, g)$ be an asymptotically flat manifold with corners along a hypersurface $\Sigma$.
Suppose $g$ has nonnegative scalar curvature away from $\Sigma$ and satisfies the mean curvature condition across $\Sigma$. 
Suppose $ M $ has nonempty  boundary $\Sh$ that satisfies 
\begin{itemize}
\item $\Sh$ has zero mean curvature;
\item $\Sh $ is strictly outer-minimizing; and
\item $  \m (g) = \frac12 \left( \frac{ |\Sh|}{  \omega_{n-1} } \right)^\frac{n-2}{n-1} $.
\end{itemize}
Then $ \Sigma$ has the same second fundamental forms in its two sides in $(M, g)$, 
$g$ is smooth across $\Sigma$, and $(M, g)$ is isometric to a spatial Schwarzschild manifold outside the horizon. 
\end{theo}

\begin{proof}
We adopt the notations in \cite{B,BL}. On $M$, 
let $g_0=g$ and define
$
g_t=u_t(x)^{\frac{4}{n-2}}g_0 ,
$
where $u_t$ is defined by
\begin{align*}
u_t(x)=1+\int_0^tv_s(x)ds
\end{align*}
and $v_t(x)$ satisfies
\be
\left\{ 
\begin{split}
\Delta_{g_0}v_t(x)= & \ 0, \ \text{outside} \  \Sht \\
v_t(x)= & \ 0,  \ \text{on and inside}  \ {\Sht}\\
\lim_{x\rightarrow \infty}v_t(x)= & -e^{-t} .
\end{split}
\right.
\ee
Here $\Sht$ is the outmost minimal area enclosure of $\Sh = \Sh^{(0)}$ in $(M, g_t)$. 
Because $\Sh$ is disjoint from $ \Sigma$ in the initial $(M, g)$ and $g$ is Lipschitz across $\Sigma$,
an examination of the proof of Theorem 2 in \cite{B} and Theorem 2.2 in \cite{BL} shows 
that $\{ g_t \}_{|t| < \epsilon } $ exists for small $\epsilon$, and 
$\Sht $ does not touch $\Sh$ for $ t > 0 $ and converges to $ \Sh $ as $t \to 0$.
As a result, for small $t$, $\Sht $ does not touch $\Sigma$ and is minimal and strictly outer-minimizing in $(M, g_t)$.

For convenience, let $M_t$ denote the exterior of $\Sht $ in $M$. 
Since $v_t$ is harmonic on $(M_t, g_0)$, as in the proof of Proposition \ref{prop-mass-capacity},
we know the restriction of $v_t$ to $ \Sigma$ is smooth,  and $v_t$ is smooth up to $ \Sigma$ from its both sides in $M$.
The same conclusion also holds for $u_t$. 

As a result, $(M_t, g_t)$ is a manifold with corners along $ \Sigma$, 
has nonnegative scalar curvature away from $ \Sigma$ and satisfies the mean curvature condition across $\Sigma$.
Since $ \Sht = \p M_t $ is outer-minimizing in $(M_t, g_t)$, the Riemmannian Penrose inequality in \cite{B, BL}, formulated 
for manifolds with corners \cite{MM16}, shows
\be \label{eq-rpi-gt}
\m (g_t) \ge \frac12 \left( \frac{ |\Sht|_{g_t}}{  \omega_{n-1} } \right)^\frac{n-2}{n-1}.
\ee
Here  $ | \Sht |_{g_t}$ is the area of $ \Sht $ in $(M, g_t)$. By Theorem 3 in \cite{B} and Lemma 2.3 in \cite{BL}, 
this area is a constant, that is
\be \label{eq-area-st}
 | \Sht |_{ g_t} = | \Sh | .
\ee

We are interested in the change of $\m (g_t)$ at $ t = 0 $. 
By Equation (113) in \cite{B} and Lemma 2.7 in \cite{BL}, 
\be
\frac{d}{d t} |_{t=0} \m (g_t) = - 2 \m (\tilde g) = \E (g) - 2 \m (g) ,
\ee
where $ \tilde g$, $ \E (g)$ are the conformally deformed metric,  the capacity constant, respectively, 
in the proof of Proposition \ref{prop-mass-capacity}. 
If $ \m (g)  >  \frac12 \E(g)$, then 
\be
\frac{d}{d t} |_{t=0} \m (g_t)  < 0. 
\ee
At $ t =0$,   it is assumed $ \m (g) = \frac12 \left( \frac{ |\Sh|}{  \omega_{n-1} } \right)^\frac{n-2}{n-1} $.  Thus,
for small $ t > 0$, 
\be
 \m (g_t ) <  \frac12 \left( \frac{ |\Sh|}{  \omega_{n-1} } \right)^\frac{n-2}{n-1}  .
\ee
This is a contradiction to  \eqref{eq-rpi-gt} and \eqref{eq-area-st}.

Therefore, $(M, g)$ satisfies $\m (g) = \frac12 \E (g)$. By the rigidity statement in Proposition \ref{prop-mass-capacity}, 
$g$ is smooth across $\Sigma$ and $(M, g)$ is isometric to a spatial Schwarzschild manifold outside its horizon. 
\end{proof}

\begin{rema}
 $\Sigma$ does not need to enclose the manifold boundary $\Sh$ in Theorem \ref{thm-rigidity}.
\end{rema}

Theorem \ref{thm-main} now follows from Theorem \ref{thm-rigidity}, because, if  $(M, g)$ is the manifold 
constructed from $(\Omega, g_{_\Omega} )$ and $ (N, g_{_N})$ in Remark \ref{rem-main}, 
then $\Sh = \p M $ is strictly outer-minimizing in $(M, g)$ by conditions (i) and (ii). Thus
Theorem \ref{thm-rigidity} applies to show $g$ is smooth and $(M, g)$ is isometric to a Schwarzschild manifold.

Corollary \ref{cor-iso-hypersurface} follows from the following theorem.

\begin{theo} \label{thm-iso-hypersurface}
Let $ \Sigma \subset (\mathbb{M}_m, g_m)  $ be a closed hypersurface enclosing the boundary $\p \mathbb{M}_m$
in a spatial Schwarzschild manifold with mass $ m >0$. Suppose $\Sigma$ is outer-minimizing. 
Let $ \tilde \Sigma \subset (\mathbb{M}_m, g_m) $ be another closed hypersurface enclosing $\p \mathbb{M}$.
Suppose $ \iota: \Sigma \to \tilde \Sigma$ is an isometry and 
$ \tilde H \circ \iota (p)  \ge H (p)$ for any $ p \in \Sigma$, 
where $H$, $\tilde H$ denote the mean curvature of $\Sigma$, $ \tilde \Sigma$, respectively. 
Then  $\tilde \Sigma= \Sigma$ up to a rigid motion of $(\mathbb{M}_m, g_m)$.
\end{theo}

\begin{proof}
Let $ \Omega$ be the finite region bounded by $\tilde \Sigma$ and $ \p \mathbb{M}_m $,
$ \p \mathbb{M}_m $ is strictly outer-minimizing in $(\Omega, g_m)$.
Let  $N$ be the unbounded region that is exterior to $ \Sigma$, then $\Sigma$ is outer-minimizing in $(N, g_m)$.
By Theorem \ref{thm-main} and Remark \ref{rem-main}, 
the manifold $(M, g)$ constructed from $(\Omega, g_m)$ and $(N, g_m)$ is smooth. 
In particular, $\Sigma$ and $\tilde \Sigma$ have the same second fundamental forms, and
$g_m$ has the same curvature quantities at $ p$ and $F(p)$, $ \forall \, p \in \Sigma$.
Writing $g_m$ in the rotationally symmetric form
$ g_m = \left( 1 - \frac{2m}{ r^{n-2} } \right)^{-1} \, d r^2 + r^2 d \sigma_o ,$
where $ \sigma_o$ is the standard round metric on the $(n-1)$-sphere, 
one knows the length square of the Ricci curvature of $g_m$ is 
$ | \mathrm{Ric} (g_m) |^2 = n (n-1) (n-2)^2 m^2 r^{-2n} $.
Since it is the same at $p$ and $F(p)$, we see that $F$ preserves $r$. From this, it is not hard to see that $F$ is the restriction of a rotation of $(\mathbb{M}_m, g_m)$.
\end{proof}

In general, the outer-minimizing condition is a global condition that is not easy to check. 
However, in a Schwarzschild manifold, a local condition that $\Sigma$ is starshaped with positive mean curvature guarantees $\Sigma$ is outer-minimizing. In particular, Theorem \ref{thm-iso-hypersurface} applies to these $\Sigma$ in $(\mathbb{M}_m, g_m)$.

\vspace{.4cm}

\noindent {\bf Acknowledgements}.  
Siyuan Lu acknowledges the support of NSERC Discovery Grant. Pengzi Miao acknowledges the support of NSF Grant DMS-1906423.

\end{document}